\documentclass[12pt]{amsart}
\usepackage{amsmath,amsthm,amsfonts,amssymb,float,mathrsfs,xypic}
\usepackage[margin=1in]{geometry}
\xyoption{arc}

\usepackage{hyperref}
\hypersetup{colorlinks}
\usepackage{color}
\definecolor{darkblue}{rgb}{0,0,0.75}
\definecolor{darkred}{rgb}{0.5,0,0}
\definecolor{darkgreen}{rgb}{0,0.5,0}
\hypersetup{ colorlinks,
citecolor=darkblue,
linkcolor=darkred,
urlcolor=darkblue,
linktocpage}


\newtheorem{thm}{Theorem}[section]  
\newtheorem*{CWS}{Theorem \cite[Stabilization Theorem]{CW}}
\newtheorem*{thmA}{Theorem A}
\newtheorem*{thmB}{Theorem B}
\newtheorem{lem}[thm]{Lemma}         
\newtheorem{prop}[thm]{Proposition}  

\theoremstyle{definition}
\newtheorem{dfn}[thm]{Definition}   
\newtheorem{rmk}[thm]{Remark}
\newtheorem{rmks}[thm]{Remarks}

\newtheorem*{clm}{Claim}

\DeclareMathOperator*{\holim}{holim}
\DeclareMathOperator*{\hocolim}{hocolim}

\newcommand{\ul}{\underline}
\newcommand{\nd}{\noindent}
\newcommand{\X}{\mathscr{X}}

\begin{document}\
\title{On the Stabilization of Embedded Thickenings}
\date{\today} 
\author{John W. Peter} 
\address{Utica College \\
Utica, NY 13502} 
\email{jwpeter@utica.edu}

\begin{abstract}  
We define a space of relative embedded thickenings of a map from a finite complex to a Poincar\'e Duality space, and show that fiberwise suspension induces a highly connected \textit{stabilization map} between such spaces. As a result, we obtain a generalization to the Poincar\'e Duality category of a smooth stabilization theorem of Connolly and Williams.

\end{abstract}

\maketitle

\setlength{\parindent}{15pt}
\setlength{\parskip}{1pt plus 0pt minus 1pt}

\newcommand{\T}{\mathcal{T}}
\newcommand{\E}{\mathbb{E}}
\def\:{\colon\thinspace}
\def\cal{\mathcal}
\newcommand{\rel}{\textit{rel}}
\newcommand{\SW}{\mathbb{S}\mathbb{W}}
\newcommand{\US}{\mathcal{S}}
\newcommand{\RS}{\Sigma}
\newcommand{\UO}{\mathcal{O}}
\newcommand{\D}{\mathcal{D}}
\newcommand{\conn}{\operatorname{conn}}
\newcommand{\hofiber}{\operatorname{hofiber}}
\newcommand{\cofiber}{\operatorname{cofiber}}

\setcounter{tocdepth}{1}
\tableofcontents
\addcontentsline{file}{sec_unit}{entry}

\section{Introduction \label{intro}}

A \textit{smooth embedding up to homotopy} of a finite complex $K$ in a manifold $M$ is given by a pair $(h,N)$ such that $N \subset M$ is a compact, codimension zero, smooth submanifold and $h: K \xrightarrow{\simeq} N$ is a homotopy equivalence. A natural question to ask is ``When does a smooth embedding up to homotopy desuspend?". Specifically, in the case that $M = S^n$, one would like to know when a smooth embedding up to homotopy of $\Sigma K$ in $S^{n+1}$ is induced by a smooth embedding up to homotopy of $K$ in $S^n$. A partial answer to this question was given by Connolly and Williams in \cite{CW} in the case that $K$ is $1$--connected. In particular, call two smooth embeddings up to homotopy $(h_0, N_0)$ and $(h_1, N_1)$ of $K$ in $S^n$ \textit{concordant} if

\begin{itemize}
\item There is an $h$--cobordism $W \subset S^n \times I$ between $N_0$ and $N_1$
\item There is a homotopy equivalence $H: K \times I \xrightarrow{\simeq} W$ extending $h_0$ and $h_1$
\end{itemize}

\noindent Writing $E(K,S^n)$ for the set of such concordance classes, Connolly and Williams prove the following: 

\begin{CWS}
\label{thm:CWS}
Assume $K$ is a homotopy finite, $r$--connected $(r \geq 1)$ complex of dimension $k \leq n-3$, $n \geq 6$. Further, assume that $k-r \leq 2$ for $n \leq 7$. Then the natural suspension map 
$$
\Sigma \: E(K,S^n) \rightarrow E(\Sigma K,S^{n+1})
$$ 

\noindent is surjective for $n\geq 2(k-r)$ and injective for $n \geq 2(k-r)+1$.
\end{CWS}

In this paper we generalize the result above to the Poincar\'e Duality category. Our generalization will follow from a Freudenthal--like stabilization theorem for our analogs of smooth embeddings up to homotopy, which we call \textit{embedded thickenings}. Roughly, an embedded thickening is specified by a map $f\:K \rightarrow X$ from a finite complex $K$ to a Poincar\'e Duality space $X$, along with a pair of spaces $(C,A)$ and a rule for gluing $K$ and $C$ together along $A$ to form $X$, up to homotopy. In the event that $X$ has a boundary $\partial X$, we define a \textit{relative embedded thickening} for a map of pairs of spaces $f\:(K,L) \rightarrow (X, \partial X)$ by specifying gluing data, similar to that given above, along with the assumption that an embedded thickening for the map $L \rightarrow \partial X$ is given. We defer the details to Definitions \ref{dfn:PDembedding} and \ref{dfn:relPDembedding} below. 

\subsection*{Main Results}
A pair $(K,L)$ satisfies $\dim(K,L) \leq k$ (or $\dim(K)\leq k$ when $L=\emptyset$) if $K$ can be obtained from $L$, up to homotopy, by attaching cells of dimension at most $k$ (that is, $(K,L)$ is weak equivalent, relative to $L$, to a relative CW complex of relative dimension at most $k$). We call $(K,L)$ \textit{homotopy finite} when $\dim(K,L)$ is finite. Let $\US_X$ denote the unreduced fiberwise suspension functor (defined below).  In what follows, we define a moduli space $\E_f(K,X\ \rel\ L)$ of relative embedded thickenings of a given map $f\:(K,L) \rightarrow (X, \partial X)$ of homotopy finite pairs, along with a \textit{stabilization map}
\[
\sigma\: \E_f(K,X\ \rel\ L) \rightarrow \E_{\US_Xf}(\US_XK, X \times D^1\ \rel\ \US_X L).
\]

\noindent Our main result concerns the connectivity of the stabilization map:

\begin{thmA}[Stabilization]
\label{thmA}
Let $f\:(K,L) \rightarrow (X,\partial X)$ be a map from a cofibration pair of homotopy finite spaces $(K,L)$, with $\dim(K,L) \leq k$, to a Poincar\'e Duality pair $(X, \partial X)$ of dimension $n$. Assume that $f\:K \rightarrow X$ is $r$--connected $(r \geq 1)$ and that $k \leq n-3$. Then the stabilization map 
\[
\sigma\: \E_f(K,X\ rel\ L) \rightarrow \E_{\US_Xf}(\US_XK, X \times D^1\rel\ \US_X L)
\] 
\noindent is $(n-2(k-r)-3)$--connected.   
\end{thmA}

\begin{rmks} (1) In the case that $K$ and $X$ are smooth manifolds, we address in \cite{KP} the extent to which a smooth embedding up to homotopy
\[
f_1\: K\to X\times D^1
\]
\textit{compresses} into a smooth embedding up to homotopy of $K$ into $X$. Theorem \hyperref[thmA]{A} plays a key role in showing that a compression exists upon the vanishing of a certain obstruction, defined using fiberwise homotopy theory, that lives in the cohomology group (cf \cite[Corollary 6.2]{KP}) 
\[
H^{2k}(K\times K; \pi_{2k-n+1}(f)\otimes\mathbb{Z}[\pi_1(X)]).
\]
Such an obstruction group appears in the work of Habegger (cf \cite{Ha}), where PL intersection theory was used to study embeddings up to homotopy. 
  \\ (2) Theorem \hyperref[thmA]{A} serves as a classification tool for embedded thickenings for $f$. In a future paper we will construct another space, $\SW_f(K,X\ \rel\ L)$, which can be thought of as a kind of moduli space of unstable fiberwise duals of embedded thickenings with underlying map $f\:(K,L) \rightarrow (X, \partial X)$. We will then establish the existence of a \textit{classification map} (cf \cite[Page 386]{CW})
\[
\theta:\E_f(K,X\ \rel\ L) \rightarrow \SW_f(K,X\ \rel\ L)
\] 
\noindent and use Theorem \hyperref[thmA]{A} to prove that this map is highly--connected. This will pave the way to an obstruction theory for the existence of embedded thickenings which, in turn, will allow us to provide enumeration results and to compute the space $\E_f(K,X\ \rel\ L)$ in some special cases.
\end{rmks}

\noindent To formulate our next main result, note that in the case $X = S^n$ and $L = \emptyset$, there is an evident ``collapse" map 
\[
\E_{\US_{S^n} f}(\US_{S^n} K, S^n\times D^1\ \rel\ \US_{S^n} \emptyset) \xrightarrow{c} \E_{\US f}(\US K,S^{n+1})
\]
\noindent where $\US$ is the usual unreduced suspension functor. Composition with the stabilization map gives the suspension map
\[
\E_f(K,S^n) \xrightarrow{c \circ \sigma} \E_{\US f}(\US K,S^{n+1}).
\] 
\noindent We will show that the map $c$ above induces a $\pi_0$ surjection and, as a consequence of Theorem \hyperref[thmA]{A}, obtain the desired generalization of the Stabilization Theorem of Connolly and Williams:

\begin{thmB}
\label{thmB}
Let $K$ be a homotopy finite complex with $\dim(K)\leq k \leq n-3$. Assume that $f\:K \rightarrow S^n$ is an $r$--connected map of spaces, $r \geq 1$. Then the induced map 
\[
\pi_0(c \circ \sigma)\: \pi_0(\E_f(K,S^n)) \rightarrow \pi_0(\E_{\US f}(\US K,S^{n+1}))
\]

\noindent is surjective for $n \geq 2(k-r)+3$ and injective for $n \geq 2(k-r)+4$.
\end{thmB}

\begin{rmk} \label{rmk:1.2} In contrast to the surgery--theoretic methods used by Connolly and Williams in \cite{CW}, the proofs of our theorems are purely homotopy--theoretic and, in particular, manifold--free. This allows us to dispense with the extra assumptions found on $n$ in the Stabilization Theorem of Connolly and Williams. We wish to emphasize that Theorem \hyperref[thmB]{B} is a generalization, and not just a PD analog, of the smooth embedding statement of Connolly and Williams. An application of the Browder--Casson--Sullivan--Wall Theorem (\cite[Theorem 5.3]{K3}) allows one to lift the result of Theorem \hyperref[thmB]{B} to the smooth category. This lifting is no small task, and is carried out in \cite{GK2}. In the appendix (Section \ref{sec:appendix} below) we sketch an argument that develops a homotopy fiber sequence relating our spaces of embedded thickenings to the Poincar\'e embedding spaces found in \cite{GK1} and \cite{GK2}.
\end{rmk}

\subsection*{Preliminaries}
\subsection*{Notation and Conventions} Let $\T$ denote the category of compactly-generated topological spaces with the Quillen model structure \cite{Q} based on weak homotopy equivalences and (Serre) fibrations. Constructions in $\T$, such as products and function spaces, will be understood to be topologized using the compactly--generated topology. In what follows, ``space" will mean ``cofibrant object of $\T$" unless otherwise stated. A space $X$ is $n$--connected if $\pi_i(X) = 0$ for every $i \leq n$ and for every choice of basepoint. In particular, a nonempty space is always $(-1)$--connected and, by convention, the empty space is $(-2)$--connected. A map of spaces $X \rightarrow Y$ is $n$--\textit{connected} if its homotopy fiber, with respect to any basepoint of $Y$, is an $(n-1)$ connected space. A weak equivalence in $\T$ is an $\infty$--connected map.\\
\hspace*{0.25in}The notation $(\overline{Y},X)$ will frequently be used to denote the pair given by the mappnig cylinder $\overline{Y}$ of a given a map $X \rightarrow Y$ with the inclusion of $X$ as $X \times 0$. Note that the pair $(\overline{Y},X)$ is a cofibration pair, ie, the inclusion $X\to\overline{Y}$ is a cofibration in $\T$.\\
\hspace*{0.25in} We will assume that the reader is familiar with homotopy (co)limits. For our purposes, a more--than--sufficient exposition of these constructions can be found in \cite[Section 0]{G}. We will also assume that the reader is familiar with the language of model categories. In particular, we point out that for any object $X$ of a model category, functorial factorization equips $X$ with a cofibrant approximation $X\xrightarrow{\sim} X^c$ and a fibrant approximation $X^f\xrightarrow{\sim} X$. The details can be found in \cite[Chapters 7 and 8]{Hi}.

\subsection*{Poincar\'e Duality Spaces} We use the definition found in \cite[Section 2.1]{K2}: Let $X$ be a homotopy finite space equipped with a local coefficient system $\mathscr{L}$ which is pointwise free abelian of rank one. Let $[X]$ denote a homology class in $H_n(X;\mathscr{L})$. The data $(X, \mathscr{L}, [X])$ equip $X$ with the structure of a \textit{Poinacr\'e Duality space of formal dimension} $n$ if cap product with $[X]$ induces an isomorphism
\[
\cap[X]\:H^*(X,\mathscr{L}) \stackrel{\cong}{\rightarrow} H_{n-*}(X;\mathscr{L} \otimes \mathscr{M})
\]
\noindent for every local coefficient system $\mathscr{M}$. We call $[X]$ the \textit{fundamental class} of $X$ and we will refer to such a space $X$ as a \textit{PD space of dimension} $n$. A cofibration pair $(X,\partial X)$ of homotopy finite spaces along with $\mathscr{L}$ and a class $[X] \in H_n(X,\partial X;\mathscr{L})$ is called a \textit{Poincar\'e Duality pair of formal dimension} $n$ if
\begin{itemize}
\item For all local systems $\mathscr{M}$, there is an induced isomorphism 
\[
\cap[X]\:H^*(X;\mathscr{L}) \stackrel{\cong}{\rightarrow} H_{n-*}(X,\partial X; \mathscr{L} \otimes \mathscr{M})
\]
\item The restriction of $\mathscr{L}$ to $\partial X$ along with the image of the fundamental class $[X]$ under the boundary homomorphism $H_n(X,\partial X; \mathscr{L}) \rightarrow H_{n-1}(\partial X; \mathscr{L})$ equips $\partial X$ with the structure of a PD space of dimension $n-1$.
\end{itemize} 
\noindent We will call such a pair $(X,\partial X)$ a \textit{PD pair of dimension} $n$.

\subsection*{Embedded Thickenings} Let $K \xrightarrow{f} X$ denote a map from a connected, homotopy finite space $K$ to a PD space $X$ or PD pair $(X, \partial X)$ of dimension $n$. The following definitions can be found in \cite[Section 2]{K5}. 
\begin{dfn}
\label{dfn:PDembedding}
An \textit{embedded thickening}\footnote{We use the term ``embedded thickening" to conform to the terminology set forth in \cite{K4}. Such embeddings were formerly called \textit{PD Embeddings} in \cite{K2}.} \textit{of f} is specified by homotopy finite spaces $A$ and $C$ along with a choice of factorization $\partial X \rightarrow C \rightarrow X$ fitting into a commutative diagram
\begin{equation}\tag{$\mathscr{D}$}
\begin{split}
\xymatrix@!0{
A \ar[rr]\ar[dd] & & C \ar[dd] & & \partial X \ar[ll]
\\
\\
K \ar[rr]_{f} & & X
}
\end{split}
\end{equation}
\noindent such that
\begin{itemize}
\item [(i)] (\textit{Stratification}) The square is $\infty$--cocartesian, ie, there is a weak homotopy equivalence of spaces $K \cup_A C \simeq X$.
\item[(ii)] (\textit{Poincar\'e Duality}) The image of the fundamental class $[X]$ under the composite
\[
H_n(X,\partial X) \cong H_n(\overline{X},\partial X) \rightarrow H_n(\overline{X},C) \cong H_n(\overline{K},A)
\]
\noindent equips $(\overline{K},A)$ with the structure of a PD pair and, similarly, the image of $[X]$ with respect to the map $H_n(X,\partial X) \rightarrow H_n(\overline{C}, \partial X \amalg A)$ equips $(\overline{C}, \partial X \amalg A)$ with the structure of a PD pair.
\item[(iii)] (\textit{Weak Transversality}) If $\dim(K) \leq k$, then the map $A \rightarrow K$ is $(n-k-1)$--connected.
\end{itemize}
\end{dfn}
\noindent We call $A$ the \textit{gluing space}, $C$ the \textit{complement}, and $f$ the \textit{underlying map} of the embedded thickening.
\begin{rmk}
\label{rmk:sphericalfibration}
(cf \cite[Remark 2.4]{K2}) If $K$ is a PD space of dimension $k\leq n-3$ then the map $A\to K$ has an $(n-k-1)$--spherical homotopy fiber over any base point (cf \cite[Lemma I.4.3]{B},\cite[Proposition 4.4]{S}) . In this sense, $A\to K$ is the analog of the normal bundle of $K\to X$. This version of embedded thickening was studied systematically by Wall in \cite[Chapter 11]{W}.
\end{rmk}

We can relativize the above as follows. Let $(K,L)$ be a cofibration pair, with $K$ and $L$ homotopy finite, and let $(X,\partial X)$ be a PD pair of dimension $n$. Fix a map $f = (f_K,f_L)\: (K,L) \rightarrow (X, \partial X)$. 

\begin{dfn}
\label{dfn:relPDembedding}
A \textit{relative embedded thickening of} $f$ consists of a commutative diagram of pairs of homotopy finite spaces
\begin{equation}\tag{$\mathscr{E}$}
\begin{split}
\xymatrix@!0{
(A_K,A_L) \ar[rrr]\ar[dd] & & & (C_K,C_L)\ar[dd]
\\
\\
(K,L) \ar[rrr]_f & & & (X, \partial X)
}
\end{split}
\end{equation}
\noindent such that
\begin{itemize}
\item [(i)] (\textit{Stratification}) Each of the associated squares of spaces
\[
\xymatrix@!0{
& & A_K\ar[rr]\ar[dd] & & C_K\ar[dd] & & & & &A_L\ar[rr]\ar[dd] & & C_L\ar[dd]
\\
(\mathscr{D}_K) & & & & & & & (\mathscr{D}_L)
\\
& & K\ar[rr]_{f_K} & & X & & & & &L\ar[rr]_{f_L} & & \partial X
}
\]
\noindent is $\infty$--cocartesian, and the square $\mathscr{D}_L$ is an embedded thickening of $f_L$.
\item[(ii)] (\textit{Poincar\'e Duality}) The image of the fundamental class $[X]$ under the composite
\begin{eqnarray*}
H_n(X,\partial X) \cong H_n(\overline{X},\partial X) &\rightarrow& H_n(\overline{X},\partial X \cup_{C_L}C_K)\\
&\cong& H_n(\overline{K},L\cup_{A_L}A_K)
\end{eqnarray*}
\noindent equips $(\overline{K},L\cup_{A_L}A_K)$ with the structure of a PD pair and, similarly, the image of $[X]$ with respect to the map 
\[
H_n(X,\partial X) \rightarrow H_n(\overline{C_K}, C_L\cup_{A_L}A_K)
\] 
\noindent equips $(\overline{C_K}, C_L\cup_{A_L}A_K)$ with the structure of a PD pair. (Here, coefficients are given by pulling back the given local system on $X$). 
\item[(iii)] (\textit{Weak Transversality}) If $\dim(K,L) \leq k$, then the map $A_K \rightarrow K$ is $(n-k-1)$--connected.
\end{itemize} 
\end{dfn}

\noindent We can picture the given Poincar\'e stratification of the pair $(X,\partial X)$ as follows:

\begin{figure}[H]
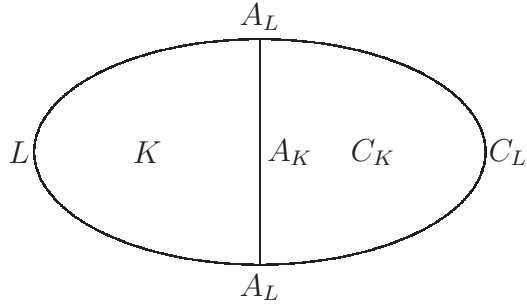

$$
\xy 
{\ellipse(30,15){}},
@={(0,18),(0,-18)},@@{*{A_L}},
@={(-32,0)},@@{*{L}},
@={(-15,0)},@@{*{K}},
@={(4,0)},@@{*{A_K}},
@={(15,0)},@@{*{C_K}},
@={(33,0)},@@{*{C_L}},
@={(0,-15),(0,15)}
, s0="prev" @@{;"prev";**@{-}="prev"}
\endxy
$$
\caption{\small{A PD Decomposition of $(X, \partial X)$}}
\label{fig:strat}
\end{figure}

\begin{rmk}
\label{rmk:concordance}
There is a notion of \textit{concordance} for embedded thickenings (cf \cite[Definition 1.3]{K5}). Let $\mathscr{D}_0$ and $\mathscr{D}_1$ denote embedded thickenings with underlying maps $f_0$,$f_1\: K \rightarrow X$ and suppose that we are given a homotopy $F\: K \times D^1 \rightarrow X$ from $f_0$ to $f_1$. Then we have an associated embedded thickening with underlying map
\[
f_0 \amalg f_1\: K \amalg K \rightarrow \partial(X \times D^1).
\]
\noindent Denote this embedding by $\mathscr{D}_{K \amalg K}$ and consider the associated map of pairs
\[
F\: (K \times D^1, K \amalg K) \rightarrow (X \times D^1, \partial(X \times D^1)).
\]
\noindent A \textit{concordance} from $\mathscr{D}_0$ to $\mathscr{D}_1$ is an embedded thickening of $F$ relative to $\mathscr{D}_{K \amalg K}$.
\end{rmk}

\subsection*{Cubical Diagrams} The language of cubical diagrams will be used throughout. We will make frequent use of the (generalized) Blakers--Massey Theorem and its dual (see \cite[Theorems 2.5, 2.6]{G}). We collect here some of the relevant terminology involving cubical diagrams, all of which can be found in \cite[Section 1]{G}. Let $\ul{n}$ denote the set $\left\{1,\ldots, n\right\}$, and write $P(\ul{n})$ for its power set, partially ordered by inclusion. Regarding $P(\ul{n})$ as a category via this partial ordering, we define an \textit{$n$--cubical diagram in $\T$} (or, briefly, an \textit{$n$--cube}) to be a functor $\X\: P(\ul{n})\rightarrow \T$. The value of the functor $\X$ at the object $S \in P(\ul{n})$ will be denoted by $\X_S$. Let $P_0(\ul{n})$ (resp, $P_1(\ul{n})$) denote the poset of nonempty subsets (resp, the poset of proper subsets) of $P(\ul{n})$. An $n$--cube $\X$ is called \textit{$j$--cartesian} if the map
\[
\X_{\emptyset} \rightarrow \holim_{S\in P_0(\ul{n})}\X_S
\]
is a $j$--connected map of spaces. The $n$--cube $\X$ is \textit{$\infty$--cartesian} if this map is a weak homotopy equivalence. Dually, an $n$--cube $\X$ is called \textit{$j$--cocartesian} if the map
\[
\hocolim_{S\in P_1(\ul{n})}\X_S\rightarrow \X_{\ul{n}}
\]
is a $j$--connected map of spaces. The $n$--cube $\X$ is \textit{$\infty$--cocartesian} if this map is a weak homotopy equivalence. If an $n$--cube is $\infty$--(co)cartesian, we will follow Goodwillie's convention and drop the ``$\infty$" from the notation. Hence, (co)cartesian is the same thing as $\infty$--(co)cartesian.\\
The $n$--cubical diagrams in $\T$ make up the objects of a category $\T^{P(\ul{n})}$ with morphisms given by natural transformations. A morphism $\mathscr{X}\to\mathscr{Y}$ in this category is a weak equivalence if for every $S\subset\ul{n}$, the map $\X_S\to\mathscr{Y}_S$ is a weak homotopy equivalence. With this notion of weak equivalence, the category $\T^{P(\ul{n})}$ carries a projective model structure in which the fibrations are defined levelwise, ie, $\X\to\mathscr{Y}$ is a fibration if for every $S\subset\ul{n}$, the map $\X_S\to\mathscr{Y}_S$ is a fibration in $\T$. For more details, see \cite[Appendix B]{GK1}.

\subsection*{Fiberwise Categories}
For a fixed map $f\: A\rightarrow B$ of spaces, let $\T(A \rightarrow B)$ denote the category whose objects are triples $(i,Y,j)$ such that $Y$ is an object of $\T$, $i\: A \rightarrow Y$ and $j\: Y \rightarrow B$ are morphisms in $\T$, and $j \circ i = f$. A morphism $(i,Y,j) \rightarrow (i',Y',j')$ is a morphism $g\: Y \rightarrow Y'$ in $\T$ such that $g \circ i = i'$ and $j' \circ g = j$. We will typically leave the structure maps out of the notation and let $Y$ denote the object $(i,Y,j) \in \T(A\rightarrow B)$.  The category $\T(A \rightarrow B)$ can be given the structure of a model category (see, eg, \cite[Chapter II, Section 2.8, Proposition 6]{Q}) whose weak equivalences and (co)fibrations are determined by the forgetful functor to $\T$. In particular, an object $(i,Y,j)$ of $\T(A \rightarrow B)$ is fibrant if $j\: Y \rightarrow B$ is a fibration in $\T$, and cofibrant if $i\: A \rightarrow Y$ is a cofibration in $\T$. Note that $\T(\emptyset \rightarrow B)$ is just the category of spaces over $B$, which we will denote by $\T(B)$. The category  $\T(B\xrightarrow{\operatorname{id}} B)$ is the category of retractive spaces (ie, ex-spaces) over $B$, which we will denote by $\mathcal{R}(B)$. An object $Y$ of $\T(A\rightarrow B)$ will be called \textit{$m$--connected} if the structure map $Y \rightarrow B$ is an $(m+1)$--connected map of spaces. The object $Y$ will be called $n$--dimensional if its fibrant approximation $Y^f$ admits a factorization $A \rightarrow Z \xrightarrow{\sim} Y^f$ such that $Z$ is obtained from $A$ by attaching cells of dimension $\leq n$.\\
We will sometimes refer to a \textit{path} between two objects of our categories, by which we mean a (zig--zag of) weak equivalence(s) between them (cf \cite[Section 7.9]{Hi}). After passing to geometric realization, such a weak equivalence gives rise to a sequence of 1--cells which are connected to each other, ie, a path in the associated space.   
\begin{rmk} Most of the categories that we will work with are not small. To avoid set--theoretic difficulties when working with such categories, we fix a Grothendieck universe $\mathcal{U}$ and use only $\mathcal{U}$--sets to form the objects of the category. In particular, we will write $|\mathscr{C}|$ for the geometric realization of the nerve of the (possibly large) category $\mathscr{C}$. This convention is not the only feasible option (see, eg, \cite[Page 9]{GK1}).
\end{rmk}

\subsection*{Fiberwise Suspension}
Given an object $K \in \T(X)$ with structure map $f\: K\to X$, the \textit{unreduced fiberwise suspension of $K$ over $X$} is the double mapping cylinder
\begin{equation*}
\US_X K = X \times -1 \cup_{f \times -1} K \times D^1 \cup_{f \times 1} X \times 1.
\end{equation*}
We regard unreduced fiberwise suspension over $X$ as a functor
\begin{equation*}
\US_X\: \T(X) \rightarrow \T(X \times S^0 \rightarrow X)
\end{equation*}
where $X \times S^0 \rightarrow X$ is the projection. It is straightforward to check that $\US_X$ maps cofibrant objects to cofibrant objects. Using the evident forgetful functor $\T(X \times S^0 \rightarrow X)\to \T(X)$, we may also regard $\US_X$ as an endofunctor of $\T(X)$ and, hence, form its $j$--fold iteration
\[
\US^j_KY = Y\times D^j\cup_{Y\times S^{j-1}}X\times S^{j-1}.
\] 
The functor $\US_X$ admits a right adjoint (cf \cite[Page 320]{K1}) 
\begin{equation*}
\mathcal{O}_X\: \T(X \times S^0\rightarrow X) \rightarrow \T(X)
\end{equation*}
given on objects by 
\begin{equation*}
Y \mapsto \holim(X \xrightarrow{i_+} Y \xleftarrow{i_-} X)
\end{equation*}
where $i_-$ and $i_+$ denote the restrictions of $X \times S^0 \rightarrow Y$ to $X\times -1$ and $X\times 1$, respectively. A straightforward application of the Blakers--Massey Theorem then proves the following:
\begin{lem}
\label{lem:Freudenthal}
Let $Y \in \T(X)$ be a cofibrant object which is $m$--connected. Then there is a morphism
\begin{equation*}
Y \rightarrow \mathcal{O}_X\US_XY 
\end{equation*}
of $\T(X)$ which is $(2m+1)$--connected. 
\end{lem}

\subsection*{Outline} In Section \ref{sec:stabmap} we construct the space of embedded thickenings for a given map $f$, along with the stabilization map. Section \ref{sec:4D} is devoted to stating and proving a technical theorem which we call the ``4D Face Theorem". This theorem is a generalization of the ``Face Theorem" of Klein \cite[Theorem 5.1]{K2} and concerns the degree to which a $2$--dimensional face of a given $4$--dimensional cubical diagram is cocartesian. Section \ref{sec:sections} contains the bulk of the argument for proving Theorem \hyperref[thmA]{A}, using fiberwise homotopy theory to construct sections to structure maps in given relative embedded thickenings. In Section \ref{sec:proofofA}, we use the maps constructed in Section \ref{sec:sections}, along with the 4D Face Theorem, to prove Theorems \hyperref[thmA]{A} and \hyperref[thmB]{B}.

\subsubsection*{Acknowledgments} This main results in this paper represent part of the authors Ph.D. thesis, which was written under the direction of Professor John R. Klein. The author is greatly indebted to Professor Klein for introducing to him and teaching him the methods of fiberwise homotopy theory employed in this paper.

\section{The Stabilization Map}
\label{sec:stabmap}

We are now equipped to define the space of relative embedded thickenings of a given map $f\: (K,L) \rightarrow (X,\partial X)$. As in Definition \ref{dfn:relPDembedding}, the data for such an embedded thickening can be encoded in a diagram of the form
\begin{equation}\tag{$\mathscr{E}$}
\begin{split}
\xymatrix@!0{
(A_K,A_L) \ar[rrr]\ar[dd] & & & (C_K,C_L)\ar[dd]
\\
\\
(K,L) \ar[rrr]_f & & & (X, \partial X)
}
\end{split}
\end{equation}
It will be useful to think of this diagram (or any diagram that represents a relative embedded thickening) as a $3$--cube:
\begin{equation}\tag{$\mathscr{E}$}
\begin{split}
\xymatrix@!0{
& A_L \ar[rr]\ar[ld]\ar'[d][dd]
& & C_L \ar[dd]\ar[ld]
\\
A_K \ar[rr]\ar[dd]
& & C_K \ar[dd]
\\
& L \ar[rr]\ar[ld]
& & \partial X \ar[ld]
\\
K \ar[rr]
& & X
}
\end{split}
\end{equation}
Let $w\T^{P(\ul{3})}$ denote the category with the same objects as $\T^{P(\ul{3})}$ but with weak equivalences as the only morphisms (recall that a weak equivalence in this case is a levelwise weak homotopy equivalence).  Then with respect to the notational convention above, we define a category $E_f(K,X\ \rel\ L)$ to be the full subcategory of $w\T^{P(\ul{3})}$ with objects those $3$--cubes that represent relative embedded thickenings of a given map $f\: (K,L) \rightarrow (X,\partial X)$. A morphism in $E_f(K,X\ \rel\ L)$ is then represented by a $4$--cube in $\T$: 
\begin{equation*}
\xymatrix@!0{
& A_L \ar[rr]\ar[ld]\ar'[d][ddd]
& & C_L \ar[ddd]\ar[ld]
& & & A_L' \ar[ld]\ar[rr]\ar'[d][ddd]
& & C_L' \ar[ld]\ar[ddd]
\\
A_K \ar[rr]\ar[rr]\ar[ddd]
& & C_K\ar[ddd]
& & & A_K' \ar[rr]\ar[ddd]
& & C_K'\ar[ddd]
\\
& & & \ar[rr]
& & & &
\\
& L \ar[rr]\ar[ld]
& & \partial X \ar[ld]
& & & L\ar[rr]\ar[ld]
& & \partial X\ar[ld]
\\
K \ar[rr]
& & X
& & & K \ar[rr]
& & X
\\
& (\mathscr{E}) & & & & & (\mathscr{E}')
}
\end{equation*}
\begin{dfn}
The \textit{space of relative embedded thickenings of} $f$, denoted by $\E_f(K,X\ \rel\ L)$, is the geometric realization of the nerve of $E_f(K,X\ \rel\ L)$. That is, 
\begin{equation*}
\E_f(K,X\ \rel\ L) = |E_f(K,X\ \rel\ L)|.
\end{equation*}
\end{dfn}
\nd Note that $\pi_0(\E_f(K,X\ \rel\ L))$ is the set of concordance classes of relative embedded thickenings of $f$ (cf Remark \ref{rmk:concordance}).
\begin{rmk}\label{rmk:deltaset} The space of relative embedded thickenings has an alternative description (see, eg, \cite[Appendix 1]{BLR} and \cite[Section 2]{GKW}) as the geometric realization of a certain fibrant $\Delta$-set (ie, simplicial set without degeneracies which satisfies the Kan condition) whose geometric realization has the weak homotopy type of $\E_f(K,X\ \rel\ L)$. For a fixed map $f\: (K,L)\rightarrow (X,\partial X)$ as in Definition \ref{dfn:relPDembedding}, a $j$-simplex of this $\Delta$--set is a commutative square of pairs of $(j+2)$-ads\footnote{For a relevant discussion of $n$--ads, see \cite[Appendix 1]{BLR} or \cite[Section 0]{W}. In particular, we will use the term $(j+2)$-ad synonymously with $(j+1)$-cube.}
\begin{equation*}
\begin{split}
\xymatrix@!0{
(A_K\times\Delta^j,A_L\times\Delta^j) \ar[rrr]\ar[dd] & & & & & (C_K\times\Delta^j,C_L\times\Delta^j)\ar[dd]
\\
\\
(K\times\Delta^j,L\times\Delta^j) \ar[rrr]_<<<{f\times id} & & & & & (X\times\Delta^j, \partial X\times\Delta^j)
}
\end{split}
\end{equation*}

\nd that satisfies the conditions of Definition \ref{dfn:relPDembedding} (suitably adjusted for the language of ads). The $i^{th}$ face map is given by sending the $j$-simplex above to the $(j-1)$-simplex
\begin{equation*}
\begin{split}
\xymatrix@!0{
(A_K\times d_i\Delta^j,A_L\times d_i\Delta^j) \ar[rrrr]\ar[dd] & & & & & & \ \ \ \ (C_K\times d_i\Delta^j,C_L\times d_i\Delta^j)\ar[dd]
\\
\\
(K\times d_i\Delta^j,L\times d_i\Delta^j) \ar[rrrr]_<<<<{f\times id} & & & & & & \ \ \ \ (X\times d_i\Delta^j, \partial X\times d_i\Delta^j)
} 
\end{split}
\end{equation*}
\nd where $d_i\Delta^j$ denotes the $i^{th}$ face of $\Delta^j$. This alternative description of our embedding space allows us to define  $\pi_j(\E_f(K,X\ \rel\ L)) = \pi_0(\E_{f \times id}(K \times D^j, X \times D^j\ \rel\ \US^j_K L))$ (cf \cite[Appendix 1, Section 6]{BLR}).
\end{rmk}
\nd Now, to define the stabilization map, let
\begin{equation}\tag{$\mathscr{D}$}
\begin{split}
\xymatrix@!0{
(A_K,A_L) \ar[rrrr]\ar[dd] & & & & (C_K,C_L)\ar[dd]
\\
\\
(K,L) \ar[rrrr]_f & & & & (X, \partial X)
}
\end{split}
\end{equation}
be an object of $E_f(K,X\ \rel\ L)$. 
Motivated by ``crossing the stratification depicted in Figure \ref{fig:strat} with the interval'' and noting that $\US_X X = X \times D^1$, define a functor
\begin{equation*}
\tilde{\sigma}\: E_f(K,X\ \rel\ L) \rightarrow E_{\US_X f}(\US_X K, X \times D^1\ \rel\ \US_X L)
\end{equation*}
on objects by
\begin{equation*}
\xymatrix@!0{
(A_K,A_L) \ar[rrrr]\ar[dd] & & & & (C_K,C_L)\ar[dd] & & & & (\US_{C_K}A_K,\US_{C_L}A_L)\ar[rrr] \ar[dd] & & & & &\  (C_K \times D^1,C_L \times D^1)\ar[dd]
\\
& & & & &\ \ \ \ \ \  \longmapsto
\\
(K,L) \ar[rrrr]_f & & & & (X, \partial X) & & & & (\US_X K, \US_X L)\ar[rrr]_{\ \ \ \ \ \US_X f} & & & & &\ \  (X \times D^1,\partial(X \times D^1))
\\
& & (\mathscr{D}) & & & & & & & & (\tilde{\sigma}\mathscr{D})
}
\end{equation*}
\begin{dfn}
The \textit{stabilization map} $\sigma$ is defined by applying the geometric realization functor to (the map of nerves induced by) $\tilde{\sigma}$:
\begin{equation*}
\sigma = |\tilde{\sigma}|\: \E_f(K,X\ \rel\ L) \rightarrow \E_{\US_X f}(\US_X K, X \times D^1\ \rel\ \US_X L).
\end{equation*}
\end{dfn}
\begin{rmk}\label{rmk:motivation}
In order to motivate the proof of Theorem \hyperref[thmA]{A}, we recall here the notion of \textit{decompressing} an embedded thickening \footnote{For a detailed account of the decompression construction, see \cite[Definition 2.3]{K5}.}. To this end, fix an object $\mathscr{D} \in E_f(K,X\ \textnormal{rel}\ L)$ (which we will also think of as the corresponding $0$-cell of the space $\E_f(K,X\ \textnormal{rel}\ L)$) and let $f^j$ denote the effect of the map $f\: K \rightarrow X$ followed by the inclusion $X \rightarrow X \times D^j$. Define a functor (called the \textit{decompression} functor)
\begin{equation*}
\tilde{\delta}: E_f(K,X\ \textnormal{rel}\ L) \rightarrow E_{f^1}(K, X \times D^1\ \textnormal{rel}\ L)
\end{equation*}
on objects by sending $\mathscr{D}$ to the relative embedded thickening
\begin{equation}\tag{$\mathscr{\tilde{\delta}D}$}
\begin{split}
\xymatrix@!0{
(\US_K A_K,\US_L A_L) \ar[rrrrr]\ar[dd] & & & & & (\US_X C_K,\US_X C_L)\ar[dd]
\\
\\
(K,L) \ar[rrrrr]_<<{\quad\quad\ \  f^1} & & & & & (X \times D^1, \partial (X \times D^1))
}
\end{split}
\end{equation}
Applying the geometric realization functor to (the map of nerves induced by) $\tilde{\delta}$ gives the \textit{decompression map}
\begin{equation*}
\delta = |\tilde{\delta}|\: \E_f(K,X\ \textnormal{rel}\ L) \rightarrow \E_{f^1}(K, X \times D^1\ \textnormal{rel}\ L).
\end{equation*}
Together, the stabilization and decompression maps give rise to the square
\begin{equation*}\tag{$\mathscr{D_{\sigma,\delta}}$}
\begin{split}
\label{diagram:Dsigmadelta}
\xymatrix@!0{
\E_f(K,X\ \rel\ L) \ar[rrrrrrr]^{\hspace*{-0.5in}\sigma}\ar[dd]_{\delta} & & & & & & & \E_{\US_Xf}(\US_XK, X \times D^1\ \rel\ \US_X L)\ar[dd]^{\delta}
\\
\\
\E_{f^1}(K,X \times D^1\ \rel\ L) \ar[rrrrrrr]_{\hspace*{-0.3in}\sigma} & & & & & & &\E_{\US_Xf^1}(\US_X K, X \times D^2\ \rel\ \US_X L)
}
\end{split}
\end{equation*}
\begin{clm} The square $\mathscr{D}_{\sigma,\delta}$ commutes up to homotopy.
\end{clm}
\nd\textit{Proof of Claim} (Sketch at the level of categories) Consider the diagram of categories whose geometric realization gives the diagram $\mathscr{D}_{\sigma,\delta}$, and let $(C_K,A_K)$ denote the complement and gluing space of a fixed embedded thickening in the upper left corner of $\mathscr{D}_{\sigma,\delta}$. Decompressing and then stabilizing gives rise to an embedded thickening with complement and gluing data \begin{equation*}\label{eqn:gs1}
(\US_{\US_XC_K}\US_XC_K,\US_{\US_XC_K}\US_KA_K)
\end{equation*}
while stabilizing and then decompressing gives rise to an embedded thickening with complement and gluing data 
\begin{equation*}\label{eqn:gs2}
(\US_{\US_XX}\US_{C_K}C_K, \US_{\US_XK}\US_{C_K}A_K)
\end{equation*} 
Since homotopy colimits commute with each other, we obtain isomorphisms between the complements and gluing spaces displayed above, respectively, by taking iterated vertical and horizontal pushouts in the diagrams
\begin{equation*}
\xymatrix@!0{
X & & C_K\ar[ll]\ar[rr] & & X & & & & X & & \ar[ll]_f\ar[rr]^fK & & X
\\
\\
X\ar@{=}[uu]\ar@{=}[dd] & & C_K\ar@{=}[uu]\ar[ll]\ar[rr]\ar@{=}[dd] & & X\ar@{=}[uu]\ar@{=}[dd] & & & & C_K\ar[uu]\ar[dd] & & A_K\ar[ll]\ar[rr]\ar[uu]\ar[dd] & & C_K\ar[uu]\ar[dd]
\\
\\
X & & C_K\ar[ll]\ar[rr] & & X & & & & X & & \ar[ll]^f\ar[rr]_fK & & X
}
\end{equation*}
A similar argument applies to all of the pairs $(X,K)$, $(\partial X,L)$, and $(C_L,A_L)$ associated with the given embedded thickening and, hence, defines an isomorphism of functors $\tilde{\delta} \circ \tilde{\sigma} \cong \tilde{\sigma} \circ \tilde{\delta}$. Passing to geometric realizations then gives the desired homotopy $\delta \circ \sigma \simeq \sigma \circ \delta$.
$\qed$  
\end{rmk}
\begin{rmk}\label{rmk:outline}
The bulk of the proof of Theorem A lies in proving that the square $\mathscr{D_{\sigma,\delta}}$ is $0$--cartesian. To do so, we will assume that we are given a path (in the lower right corner of $\mathscr{D_{\sigma,\delta}}$) between the images of two embedded thickenings 
\begin{equation*}
\xymatrix@!0{
& & (A'_K,A'_L) \ar[rrrr]\ar[dd] & & & & (C'_K,C'_L)\ar[dd(0.8)] & & & & & (A''_K,A''_L)\ar[rrrr]\ar[dd] & & & & & (C''_K,C''_L)\ar[dd]
\\
\\
& & (\US_X K, \US_XL) \ar[rrr(0.9)] & & & & &(X \times D^1,\partial(X\times D^1)) & & & & (K,L) \ar[rrr(0.9)] & & & & & (X \times D^1,\partial(X\times D^1))
\\
& & & & (\mathscr{A}') & & & & & & & & & (\mathscr{A}'')
}
\end{equation*}  
coming from the upper right and lower left corners of $\mathscr{D_{\sigma,\delta}}$, respectively, and show that this path, along with certain ``section data" (cf Section \ref{sec:sections}) gives rise to the ``punctured" $4$--cube
\begin{equation*}\tag{$\mathscr{C}$}
\begin{split}
\xymatrix@!0{
& 
& & & C'_K \ar[ddd]\ar[ld]
& & & C'_K \ar[ld]\ar[rrr]\ar[ddd]
& & & A'_K \ar[ld]\ar[ddd]
\\
K \ar[rrr]\ar[ddd]
& & & X \ar[ddd]
& & & X \ar[rrr]\ar[ddd]
& & & \US_X K\ar[ddd]
\\
& & & & & \xrightarrow{\hspace*{0.5in}}
\\
& K \ar[rrr]\ar[ld]
& & & X \ar[ld]
& & & X \ar[rrr]\ar[ld]
& & & \US_X K\ar[ld]
\\
A''_K \ar[rrr]
& & & C''_K
& & & C''_K \ar[rrr]
& & & \US_{\US_X K} A'_K
}
\end{split}
\end{equation*}
Inserting the homotopy limit of this $4$--cube in the upper left corner will produce a $4$--cube that satisfies the hypotheses of the 4D Face Theorem (cf Section \ref{sec:4D}) and, thus, produce the desired embedded thickening in the upper left corner of $\mathscr{D}_{\sigma,\delta}$. 
\end{rmk}


\section{The 4D Face Theorem}
\label{sec:4D}

The following theorem is an essential ingredient in the proof of Theorem \hyperref[thmA]{A}, and concerns the degree to which a certain face of a given 4--cube is cocartesian. 

\begin{thm}\label{thm:4D} (4D Face Theorem)
Let $X\: P(\textnormal{\ul{4}}) \rightarrow \T$ be the $4$--dimensional cubical diagram of spaces represented by the commutative diagram
\begin{equation*}
\xymatrix@!0{
& X_{\emptyset} \ar[rr]\ar[ld]\ar'[d][ddd]
& & X_3 \ar[ddd]\ar[ld]
& & & X_4 \ar[ld]\ar[rr]\ar'[d][ddd]
& & X_{34} \ar[ld]\ar[ddd]
\\
X_1 \ar[rr]\ar[rr]\ar[ddd]
& & X_{13}\ar[ddd]
& & & X_{14} \ar[rr]\ar[ddd]
& & X_{134}\ar[ddd]
\\
& & & \ar[rr]
& & & &
\\
& X_2 \ar[rr]\ar[ld]
& & X_{23} \ar[ld]
& & & X_{24}\ar[rr]\ar[ld]
& & X_{234}\ar[ld]
\\
X_{12} \ar[rr]
& & X_{123}
& & & X_{124} \ar[rr]
& & X_{1234} }
\end{equation*}
Assume that 
\begin{itemize}
\item The $4$--cube $X$ is cartesian 
\item The spaces $X_S$ are connected for each nonempty $S \subset \textnormal{\ul{4}}$
\item Each $3$--dimensional face which meets $X_{1234}$ is strongly cocartesian (ie, every $2$--dimensional face is cocartesian)
\item Each map $X_S \rightarrow X_{S \cup \left\{i\right\}}$ is $k_i$--connected for $S$ and $\left\{i\right\}$ nonempty subsets of $\textnormal{\ul{4}}$, $i \notin S$.
\item  $k_i, k_j \geq 2$ for some $i \neq j$.
\end{itemize}
Then each of the squares 
\begin{equation*}
\xymatrix@!0{
X_{\emptyset}\ar[rr]\ar[dd] & & X_j\ar[dd]
\\
\\
X_i\ar[rr] & & X_{ij}
}
\end{equation*}
is $\left(\sum_{i=1}^4k_i-1\right)$-cocartesian for $1\leq i < j \leq 4$.
\end{thm}
\begin{rmk}\label{rmk:4Drmk} Following \cite[Notation 1.12]{G}, for $U \subset T \subset \ul{4}$ let
\begin{equation*}
\partial^T_U X = \left\{V \mapsto X(V \cup U)\ :\ V \subset T-U\right\}
\end{equation*}
denote the $(|T|-|U|)$-subcube of $X$. Then $\partial^{\ul{4}}_{\ul{4}-T} X$ is the $|T|$--subcube of $X$ terminating in $X_{1234}$. Suppose each of these $|T|$--subcubes is $k(T)$--cartesian. One can easily check that $\textnormal{min}\left\{\sum_{\alpha}k(T_{\alpha})\right\} = \sum_{i=1}^4k_i-2$, where the minimum is taken over all partitions $\left\{T_{\alpha}\right\}_{\alpha}$ of $\ul{4}$ by nonempty subsets. By the dual Blakers-Massey Theorem \cite[Theorem 2.6]{G} $X$ is $\left(\sum_{i=1}^4k_i+1\right)$--cocartesian. Write $X$ as a map of $3$--cubes $Y \rightarrow Z$. By hypothesis, $Z$ is strongly cocartesian. Let $H_*(X)$ denote the reduced homology of the total cofiber of $X$, and similarly for $Y$ and $Z$. Then $H_n(Z)=0$ for all $n$ and $H_n(X)=0$ for $n \leq \sum_{i=1}^4k_i+1$. From the long exact sequence
\begin{equation*}
\cdots \rightarrow H_n(Z) \rightarrow H_n(X) \rightarrow H_{n-1}(Y) \rightarrow H_{n-1}(Z) \rightarrow \cdots
\end{equation*}
we conclude that $H_n(Y)=0$ for $n \leq \sum_{i=1}^4k_i$. This result motivates the first claim made in the proof below.
\end{rmk}

\begin{proof}[Proof of 4D Face Theorem] As above, let $X$ denote the $4$--cube. Without loss in generality, assume that all of the maps in $X$ are fibrations. Using Remark \ref{rmk:4Drmk} and an argument similar to that given in the proof of \cite[Theorem 5.1]{K2}, our final hypothesis in the statement of the theorem guarantees that $X_{\emptyset}$ is nonempty and connected.
\begin{clm}\label{clm:3D} Each of the $3$--dimensional subcubical diagrams 
\begin{equation*}
\xymatrix@!0{
& X_{\emptyset} \ar[rr]\ar[ld]\ar'[d][dd]
& & X_k \ar[dd]\ar[ld]
\\
X_i \ar[rr]\ar[rr]\ar[dd]
& & X_{ik} \ar[dd]\ar[dd]
\\
& X_j \ar[rr]\ar[ld]
& & X_{jk} \ar[ld]
\\
X_{ij} \ar[rr]
& & X_{ijk}}
\end{equation*}
is $\sum_{i=1}^4k_i$--cocartesian for $1 \leq i < j < k \leq 4$.
\end{clm}
\nd\textit{Proof of Claim} : Choose one of the $3$--cubes meeting $X_{\emptyset}$, say
\begin{equation}\tag{$\partial^{\ul{3}}_{\emptyset}X$}
\begin{split}
\xymatrix@!0{
& X_{\emptyset} \ar[rr]\ar[ld]\ar'[d][dd]
& & X_3 \ar[dd]\ar[ld]
\\
X_1 \ar[rr]\ar[rr]\ar[dd]
& & X_{13} \ar[dd]
\\
& X_2 \ar[rr]\ar[ld]
& & X_{23} \ar[ld]
\\
X_{12} \ar[rr]
& & X_{123}}
\end{split}
\end{equation}
It will be enough to prove the claim for this $3$--cube. Recall that $P_1(\ul{n}) = P(\ul{n}) - \ul{n}$. The homotopy colimit of the punctured $4$--cube $X|_{P_1(\ul{4})}$ is the iterated homotopy colimit
\begin{equation*}
\hocolim\left(X_{123} \leftarrow \hocolim_{U\subset P_1(\ul{3})}X_U \rightarrow \hocolim_{U\subset P_1(\ul{4}) - P(\ul{3})}X_U\right)
\end{equation*}
and, thus, we can form the following commutative diagram, in which the square is cocartesian. 
\begin{equation}\tag{2.1}
\begin{split}
\xymatrix@!0{
\hocolim_{U\subset P_1(\ul{3})}X_U \ar[rrrrr]\ar[dd] & & & & & \hocolim_{U\subset P_1(\ul{4}) - P(\ul{3})}X_U\ar[dd]^{\beta}\ar@/^1pc/[rrrddd]^-{\sim}
\\
\\
X_{123}\ar[rrrrr]_-{\gamma}\ar@/_1pc/[rrrrrrrrd] & & & & & \hocolim_{U\subset P_1(\ul{4})}X_U\ar[rrrd]^-{\alpha}
\\
& & & & & & & & X_{1234}
}
\end{split}
\end{equation}
The equivalence labeled above arises from the assumption that every $3$--cube meeting $X_{1234}$ is strongly cocartesian. The canonical map $\alpha$ is  $\left(\sum_{i=1}^4k_i+1\right)$--connected since $X$ is  $\left(\sum_{i=1}^4k_i+1\right)$--cocartesian. Thus, by \cite[Proposition 1.5(ii)]{G} the map $\beta$ is $\left(\sum_{i=1}^4k_i\right)$--connected. Moreover, by hypothesis, the map $X_{123}\rightarrow X_{1234}$ is $k_4$--connected. Hence, another application of \cite[Proposition 1.5(ii)]{G} shows that the bottom horizontal map $\gamma$ is $k_4$--connected. To prove the claim, we have to show that the left vertical map in $(2.1)$ is $\left(\sum_{i=1}^4k_i\right)$-connected. By \cite[Lemma 5.6(2)]{K2} it will be enough to show that the top horizontal map in $(2.1)$ is $2$--connected. To this end, note that each of the spaces in $(2.1)$ admits a map to $X_{1234}$. Taking homotopy fibers of the maps of each space in the square in $(2.1)$ over $X_{1234}$ (with respect to any choice of basepoint in $X_{1234}$) gives rise to the square of homotopy fibers  
 \begin{equation}\tag{2.2}
\begin{split}
\xymatrix@!0{
\hofiber(\hocolim_{U\subset P_1(\ul{3})}X_U\rightarrow X_{1234}) \ar[rrrrrrr]\ar[dd] & & & & & & & \ast \ar[dd]
\\
\\
\hofiber(X_{123} \rightarrow X_{1234})\ar[rrrrrrr] & & & & & & & \hofiber(\hocolim_{U\subset P_1(\ul{4})}X_U\rightarrow X_{1234})
}
\end{split}
\end{equation}
This square is cocartesian (see \cite[Chapter 3]{MV}). Let $s$ denote the connectivity of the space in the upper left corner of $(2.2)$. Then the claim will follow if we can show that $s\geq 1$. To this end, let $C_{top}$ denote the homotopy cofiber of the top horizontal map in $(2.2)$. Then $C_{top}$ is $(s+1)$--connected. Use the $k_4$--connected map $\gamma$ from $(2.1)$ to form the square
\begin{equation}\tag{2.3}
\begin{split}
\xymatrix@!0{
X_{123}\ar[rrrr]\ar[dd]_{\gamma} & & & & X_{1234}\ar[dd]^=
\\
\\
\hocolim_{U\subset P_1(\ul{4})}X_U\ar[rrrr] & & & & X_{1234}
}
\end{split}
\end{equation}
The induced map of horizontal homotopy fibers in $(2.3)$ is then $k_4$--connected. But this induced map is precisely the bottom horizontal map in $(2.2)$. Thus the homotopy cofiber of the bottom horizontal map in $(2.2)$ is $k_4$--connected. Call this cofiber $C_{bottom}$. Since $(2.2)$ is cocartesian, we have a weak equivalence $C_{top} \xrightarrow{\sim} C_{bottom}$. This implies that $s+1 = k_4$. By hypothesis, we may assume that $k_4 \geq 2$, so that $s \geq 1$, and the claim follows. 

Now we prove the statement concerning the degree to which each $2$--face meeting $X_{\emptyset}$ is cocartesian. Choose one of these $2$--faces:
\begin{equation}\tag{$\partial^{\ul{2}}_{\emptyset}X$}
\begin{split}
\xymatrix@!0{
X_{\emptyset}\ar[rr]\ar[dd] & & X_1\ar[dd]
\\
\\
X_2\ar[rr] & & X_{12}
}
\end{split}
\end{equation}
It will be enough to show that $\partial^{\ul{2}}_{\emptyset}X$ is $(\sum_{i=1}^4k_i-1)$--cocartesian. By hypothesis, the $3$--cube $\partial^{\ul{4}}_{\left\{3\right\}}$ is strongly cocartesian. Thus, the face
\begin{equation}\tag{$\partial^{\ul{3}}_{\left\{3\right\}} X$}
\begin{split}
\xymatrix@!0{
X_3 \ar[rr]\ar[dd] & & X_{13}\ar[dd]
\\
\\
X_{23}\ar[rr] & & X_{123}
}
\end{split}
\end{equation}
is cocartesian. As in the claim above, there is a commutative diagram
\begin{equation}\tag{2.4}
\begin{split}
\xymatrix@!0{
\hocolim_{U\subset P_1(\ul{2})}X_U \ar[rrrrr]\ar[dd] & & & & & \hocolim_{U\subset P_1(\ul{3}) - P(\ul{2})}X_U\ar[dd]\ar@/^1pc/[rrrddd]^{\sim}
\\
\\
X_{12}\ar[rrrrr]\ar@/_1pc/[rrrrrrrrd] & & & & & \hocolim_{U\subset P_1(\ul{3})}X_U\ar[rrrd]
\\
& & & & & & & & X_{123}
}
\end{split}
\end{equation}
The theorem follows by verifying that the left vertical map in $(2.4)$ is $\left(\sum_{i=1}^4k_i-1\right)$--connected. The argument is almost identical to the one just given. We omit the details. 
\end{proof}

\section{Section Data for Embedded Thickenings}
\label{sec:sections}
In this section, we construct the maps necessary to form the punctured $4$--cube $\mathscr{C}$ displayed at the end of Remark \ref{rmk:outline}. First, we need a lemma.
\begin{lem}
\label{lem:hofiber}
Let $g\: X \rightarrow Y$ be a $t$--connected map of connected, based spaces with $\conn(X) = \conn(Y) = s$. Then the square
\begin{equation*}
\xymatrix@!0{
X \ar[rr] \ar[dd]_g & & \Omega\Sigma X \ar[dd]
\\
\\
Y \ar[rr]  & & \Omega\Sigma Y
}
\end{equation*}
is $(s+t)$-cartesian.
\end{lem}
\begin{proof} Recall that for a based space $Z$ with basepoint $\ast$, the homotopy fiber of the inclusion $\ast \hookrightarrow Z$ is the based loop space $\Omega Z$. The square in question is then determined by taking the homotopy fibers of the horizontal maps in the cocartesian $3$--cube
\begin{equation}\tag{$\mathscr{X}$}
\begin{split}
\xymatrix@!0{
& X \ar[rr]\ar[ld]\ar'[d][dd]
& & \ast \ar[dd]\ar[ld]
\\
\ast \ar[rr]\ar[dd]
& & \Sigma X \ar[dd]
\\
& Y \ar[rr]\ar[ld]
& & \ast \ar[ld]
\\
\ast \ar[rr]
& & \Sigma Y
}
\end{split}
\end{equation}
For $T$ a nonempty subset of $\left\{1,2,3\right\}$, let $k(T)$ denote the degree to which the face $\partial^T_{\emptyset}\mathscr{X}$ is cocartesian. Labeling the single points in the top square of $\mathscr{X}$ by $1$ and $3$, and labeling $Y$ by $2$, one can easily check that 
\begin{gather*}
k(\left\{1\right\}) = k(\left\{3\right\}) = s +1\\
k(\left\{1,2\right\}) = k(\left\{2,3\right\}) = t+1\\ 
k(\left\{2\right\}) = t\\
\tag*{\text{and}}k(\left\{1,3\right\}) = k(\left\{1,2,3\right\}) = \infty.
\end{gather*}
By the generalized Blakers-Massey Theorem \cite[Theorem 2.5]{G}, $\mathscr{X}$ is $(1-3+s+t+2) = (s+t)$--cartesian. An application of \cite[Proposition 1.18]{G} completes the proof.
\end{proof}
For the rest of this section, it will be convenient to recall that the fiberwise suspension of a space over itself is equal to taking the product of the space with $D^1$. Thus, as in Remark \ref{rmk:outline}, assume that we are given relative embedded thickenings
\begin{equation*}
\xymatrix@!0{
& & (A'_K,A'_L) \ar[rrrr]\ar[dd] & & & & & (C'_K,C'_L)\ar[dd] & & & & (A''_K,A''_L)\ar[rrr]\ar[dd] & & & & (C''_K,C''_L)\ar[dd]
\\
\\
& & (\US_X K, \US_XL) \ar[rrrr(0.85)]_{\qquad \US_Xf} & & & & & (\US_XX,\US_X\partial X) & & & & (K,L) \ar[rrr(0.85)]_{\hspace*{0.2in}f^1} & & & & (\US_XX,\US_X\partial X)
\\
& & & & (\mathscr{A}') & & & & & & & & &(\mathscr{A}'')
}
\end{equation*}
with $\mathscr{A}'$ and $\mathscr{A}''$ in the upper right and lower left corners, respectively, of the square $\mathscr{D}_{\sigma,\delta}$ of Remark \ref{rmk:motivation}. Recall that we are working relative to a given embedded thickening with underlying map $L\to\partial X$. Thus we may assume that the diagrams $\mathscr{A}'$ and $\mathscr{A}''$ restrict along $\US_XL$ and $L$, respectively, to the diagrams of stabilization and decompression of the given embedded thickening of $\L\to\partial X$. So, according to \cite[Section 4]{K5}, there is a PD pair $(C_L, A_L)$, along with an object $C_K \in \T(C_L \rightarrow X)$ such that $\mathscr{A}'$ and $\mathscr{A}''$ are weakly equivalent to the following embedded thickenings (which we still call $\mathscr{A}'$ and $\mathscr{A}''$).
\begin{equation*}
\xymatrix@!0{
& & (A'_K,\US_{C_L}A_L) \ar[rrrr(0.75)]\ar[dd] & & & & & (\US_{C_K}C_K,\US_{C_L}C_L)\ar[dd] & & & & (A''_K,\US_LA_L)\ar[rrr(0.8)]\ar[dd] & & & & (\US_XC_K,\US_XC_L)\ar[dd]
\\
\\
& & (\US_X K, \US_XL) \ar[rrrr(0.85)]_{\qquad \US_Xf} & & & & & (\US_XX,\US_X\partial X) & & & & (K,L) \ar[rrr(0.825)]_{\hspace*{0.2in}f^1} & & & & (\US_XX,\US_X\partial X)
\\
& & & & (\mathscr{A}') & & & & & & & & &(\mathscr{A}'')
}
\end{equation*}
Applying the decompression map to $\mathscr{A}'$ gives 
\begin{equation*}\tag{$\delta\mathscr{A}'$}
\begin{split}
\xymatrix@!0{
& & (\US_{\US_X K} A'_K, \US_{\US_X L}A'_L)\ar[rrrr]\ar[dd] & & & & & & &(\US_{\US_XX} \US_{C_K}C_K, \US_{\US_XX} \US_{C_L}C_L)\ar[dd(0.8)] 
\\
\\
& & (\US_X K, \US_X L)\ar[rrrrrr(0.9)]_{\hspace*{0.3in}(\US_Xf)^1}  & & & & & & & (\US_X^2X,\US_X^2\partial X) & & & & & 
}
\end{split}
\end{equation*}
and applying the stabilization map to $\mathscr{A}''$ gives
\begin{equation*}\tag{$\sigma\mathscr{A}''$}
\begin{split}
\xymatrix@!0{
(\US_{\US_X C_K} A''_K, \US_{\US_X C_L} \US_LA_L)\ar[rrrrr(0.8)]\ar[dd] & & & & & & & (\US_{\US_X C_K}\US_X C_K,\US_{\US_X C_L}\US_X C_L)\ar[dd]
\\
\\
(\US_X K, \US_X L)\ar[rrrrrr(0.9)]_{\hspace*{0.6in}\US_{\US_XX} f} & & & & & & & (\US_X^2X,\US_X^2\partial X)
}
\end{split}
\end{equation*}
\begin{lem}
\label{lem:sectionC}
Assume that there is a path from $\delta\mathscr{A}'$ to $\sigma\mathscr{A}''$ in the category $E_{\US_Xf^1}(\US_X K,\US^2_X X\ \rel\ \US_X L)$. Assume further that $f\: K \rightarrow X$ is $r$--connected $(r \geq 1)$, $\dim(K,L) \leq k \leq n-3$, and $n \geq 2(k-r) +2$. Then there is a lift
\begin{equation*}
\xymatrix@!0{
& & & A'_K \ar[dd]
\\
\\
C_K\times S^0\ar[rr]\ar@{-->}[rrruu] & & & \US_{C_K}C_K 
}
\end{equation*}
where the map $A_K'\to \US_{C_K}C_K$ is the top horizontal map in $\mathscr{A}'$.
\end{lem}
\proof
Since $f$ is $r$--connected, the natural map $\US_XK \rightarrow X$ is $(r+1) \geq 2$--connected. Hence, every local coefficient system on $\US_XK$ arises by pullback from one on $X$. An easy argument using a relative Mayer-Vietoris sequence (with coefficients in any local system) shows that $\dim(\US_X K, \US_X L) \leq k+1$. So, by definition, the map $A'_K \rightarrow \US_X K$ in $\mathscr{A}'$ is $(n+1)-(k+1)-1 = (n-k-1)$--connected. Since $k \leq n-3$, we know that $n-k-1 \geq 2$ and hence, by \cite[Lemma 5.6 (2)]{K2}, we infer that map $A'_K \rightarrow \US_{C_K}C_K$ is $(r+1)$--connected. Thus, the Blakers-Massey Theorem implies that the square associated with $\mathscr{A}'$
\begin{equation}\tag{$\mathscr{A}'$}
\begin{split}
\xymatrix@!0{
A'_K \ar[rrr]\ar[dd] & & & \US_{C_K}C_K\ar[dd]
\\
\\
\US_X K \ar[rrr]_{\US_Xf} & & & \US_XX
}
\end{split}
\end{equation}
is $(n-k+r-1)$--cartesian. Using Lemma \ref{lem:Freudenthal}, what we have so far assembles into the square  
\begin{equation}\tag{$\mathscr{F}$}
\begin{split}
\xymatrix@!0{
A'_K\ar[rrrrr(0.85)]\ar[dd] & & & & & & \mathcal{O}_{\US_XK}\US_{\US_XK}A'_K\ar[dd(0.9)]
\\
\\
\US_XK \times_{\US_XX} \US_{C_K}C_K \ar[rrrrr(0.7)] & & & & & & & \mathcal{O}_{\US_XK}\US_{\US_XK}(\US_XK \times_{\US_XX} \US_{C_K}C_K)
}
\end{split}
\end{equation}
of objects of $\T(\US_XK)$ with left vertical map $(n-k+r-1)$--connected. Let 
\begin{gather*}
F = \hofiber(A_K' \rightarrow \US_XK)\\
\tag*{\text{and}}F' = \hofiber(\US_XK \times_{\US_XX} \US_{C_K}C_K \rightarrow \US_XK)
\end{gather*}
where the homotopy fibers are taken with respect to any basepoint in $\US_X K$. The spaces $F$ and $F'$ are both $(n-k-2)$--connected, and the map $F \xrightarrow{h} F'$ induced by the square $\mathscr{F}$ is $(n-k+r-1)$--connected. The map $h$ fits into the square of homotopy fibers of $\mathscr{F}$ over $\US_XK$ (with respect to any basepoint of $\US_X K$) which is given by
\begin{equation}\tag{$h\mathscr{F}$}
\begin{split}
\xymatrix@!0{
F \ar[rr]\ar[dd]_h & & \Omega\Sigma F\ar[dd]
\\
\\
F' \ar[rr] & & \Omega\Sigma F'
}
\end{split}
\end{equation}
By Lemma \ref{lem:hofiber}, $h\mathscr{F}$ is $(2n - 2k + r - 3)$--cartesian and, by \cite[Proposition 1.18]{G} so is $\mathscr{F}$. Now, form the composite
\begin{equation}
\US_{\US_XK}(C_K\times S^0) \rightarrow \US_X(C_K\times S^0) = \US_XC_K\times S^0 = \US_{\US_XC_K}\emptyset \rightarrow \US_{\US_XC_K}A''_K \simeq \US_{\US_XK}A'_K
\end{equation}
where the last weak equivalence comes from our assumption that there is a path from $\delta\mathscr{A}'$ to $\sigma\mathscr{A}''$. Apply the functor $\mathcal{O}_{\US_XK}$ and compose with the map from Lemma \ref{lem:Freudenthal} to get
\begin{equation}\label{equation:(4)}
C_K\times S^0\rightarrow  \mathcal{O}_{\US_XK}\US_{\US_XK}(C_K\times S^0) \rightarrow \mathcal{O}_{\US_XK}\US_{\US_XK}A'_K.
\end{equation}
The composite (\ref{equation:(4)}), the square $\mathscr{F}$, and the evident map
\begin{equation*}
C_K\times S^0 = \US_{\US_XX} \emptyset \times_{\US_XX} C_K \simeq \US_{\US_XX} \emptyset \times_{\US_XX}\US_{C_K}C_K \rightarrow S_X K \times_{\US_XX}\US_{C_K}C_K
\end{equation*}
combine to form the following diagram of spaces over $\US_X K$: 
\begin{equation*}
\xymatrix@!0{
C_K\times S^0 \ar@/^1pc/[ddrrrrrrr]\ar@/_/[ddddrr]\ar@{-->}[ddrr]
\\
\\
& & A'_K \ar[rrrrr]\ar[dd] & & & & & \mathcal{O}_{\US_XK}\US_{\US_XK}A'_K\ar[dd(0.8)]
\\
\\
& & S_X K \times_{\US_XX}\US_{C_K}C_K \ar[rrrrr(0.7)] & & & & & & & \mathcal{O}_{\US_XK}\US_{\US_XK}(S_X K \times_{\US_XX}\US_{C_K}C_K)
}
\end{equation*}
By obstruction theory, the dashed arrow exists provided that $\dim(C_K) \leq 2n-2k+r-3$. Using duality and excision, we have isomorphisms 
\begin{equation*}
H^*(C_K) \cong H_{n+1-\ast}(\overline{C_K},A'_K) \cong H_{n+1-\ast}(\overline{\US_XX}, \US_XK).
\end{equation*}
for all local coefficient systems. Since $(\US_X X, \US_XK)$ is an $(r+1)$--connected pair of spaces, the isomorphism above implies  that $C_K$ is cohomologically $(n-r-1)$--dimensional (ie, its cohomology vanishes in degrees $>n-r-1$).  But $r \leq k \leq n-3$, so that $n-r-1 \geq 2$. Hence, by \cite[Proposition 8.1]{GK1}, $\dim(C_K) \leq n-r-1$. Thus, the dashed arrow exists provided that $n-r-1 \leq 2n - 2k + r - 3$, which is equivalent to $n \geq 2(k-r)+2$. This establishes the existence of the map 
\begin{equation*}
C_K \times S^0 \rightarrow A'_K \rlap{\hspace{1.75in} \qedsymbol}
\end{equation*}
We now make a similar construction associated with the square $\mathscr{A}''$.
\begin{lem}
\label{lem:sectionK}
With the assumptions of the previous lemma, there is a lift
\begin{equation*}
\xymatrix@!0{
& & & A''_K \ar[dd]
\\
\\
K\times S^0\ar[rrr]_{\text{proj.}}\ar@{-->}[rrruu] & & & K 
}
\end{equation*}
where the map $A''_K\to K$ is the left vertical map in the diagram $\mathscr{A}''$.
\end{lem}
\proof
This is similar to the proof of the previous lemma. The Blakers-Massey Theorem implies that the square
\begin{equation*}
\xymatrix@!0{
A''_K \ar[rr]\ar[dd] & & \US_X C_K\ar[dd]
\\
\\
K\ar[rr] & & \US_XX
}
\end{equation*}
associated with $\mathscr{A}''$ is $(n-k+r-1)$--cartesian. Thus we have an $(n-k+r-1)$--connected  map
\begin{equation*}
A''_K \rightarrow K \times_{\US_XX} \US_XC_K
\end{equation*}
Using Lemma \ref{lem:Freudenthal}, form the following commutative square of objects in $\T(\US_X C_K)$: 
\begin{equation}\tag{$\mathscr{G}$}
\begin{split}
\xymatrix@!0{
A''_K\ar[rrrrr]\ar[dd] & & & & & \mathcal{O}_{\US_XC_K}\US_{\US_XC_K}A''_K\ar[dd]
\\
\\
K \times_{\US_XX} \US_XC_K \ar[rrrrr(0.5)] & & & & & & \mathcal{O}_{\US_XC_K}\US_{\US_XC_K}(K \times_{\US_XX} \US_XC_K)
}
\end{split}
\end{equation}
Map every space in $\mathscr{G}$ to $\US_XC_K$ and take homotopy fibers (for any choice of basepoint) to get
\begin{equation}\tag{$h\mathscr{G}$}
\begin{split}
\xymatrix@!0{
G = \hofiber(A''_K \rightarrow \US_XC_K) \ar[rrrrrr]\ar[dd]_{h'} & & & & & & \Omega\Sigma G\ar[dd]
\\
\\
G' = \hofiber(K \times_{\US_XX} \US_XC_K \rightarrow \US_XC_K) \ar[rrrrrr] & & & & & & \Omega\Sigma G'
}
\end{split}
\end{equation}
in which $\conn(G) = \conn(G') = r-1$ and $\conn(h') =  (n-k+r-1)$. By Lemma \ref{lem:hofiber}, $h\mathscr{G}$ is $(n - k + 2r - 2)$--cartesian and, hence, so is $\mathscr{G}$. Form the composite
\begin{equation*}
\US_{\US_XC_K}(K\times S^0) \rightarrow \US_X(K\times S^0) = \US_{\US_XK} \emptyset \rightarrow \US_{\US_XK}A'_K \simeq \US_{\US_XC_K}A''_K.
\end{equation*}
in $\T(\US_X C_K)$, where the last weak equivalence comes from the assumption that there is a path from $\delta\mathscr{A}'$ to $\sigma\mathscr{A}''$. As above, apply the functor $\mathcal{O}_{\US_XC_K}$ to get a map
\begin{equation*}
\mathcal{O}_{\US_XC_K}\US_{\US_XC_K}(K\times S^0) \rightarrow \mathcal{O}_{\US_XC_K}\US_{\US_XC_K}A''_K.
\end{equation*}
and compose with the map from $K\times S^0$ provided by Lemma \ref{lem:Freudenthal} to get
\begin{equation*}
K\times S^0 \rightarrow \mathcal{O}_{\US_XC_K}\US_{\US_XC_K}A''_K.
\end{equation*}
Combine this with the map $K\times S^0 = K \times_{\US_XX} \US_{\US_XX} \emptyset \rightarrow K \times_{\US_XX} \US_{\US_XX} C_K \simeq K \times_{\US_XX}\US_XC_K$ and the square $\mathscr{G}$ to form the diagram 
\begin{equation*}
\xymatrix@!0{
K\times S^0 \ar@/^1pc/[ddrrrrrrr]\ar@/_/[ddddrr]\ar@{-->}[ddrr]
\\
\\
& & A''_K \ar[rrrrr]\ar[dd] & & & & & \mathcal{O}_{\US_XC_K}\US_{\US_XC_K}A''_K\ar[dd(0.9)]
\\
\\
& & K \times_{\US_XX} \US_XC_K \ar[rrrrr(0.7)] & & & & & & &\mathcal{O}_{\US_XC_K}\US_{\US_XC_K}(K \times_{\US_XX} \US_XC_K)
}
\end{equation*}
Again, by obstruction theory, the dashed arrow exists provided that $k \leq n-k+2r-2$, which is equivalent to $n \geq 2(k-r)+2$. By hypothesis, we have the desired map
\begin{equation*}
K \times S^0 \rightarrow A''_K \rlap{\hspace{1.75in} \qedsymbol}.
\end{equation*}

\section{Proofs of the Stabilization Theorems \hyperref[thmA]{A} and \hyperref[thmB]{B}}\label{sec:proofofA}
We are finally in a position to determine the conditions under which the square $\mathscr{D}_{\sigma,\delta}$  is $0$--cartesian (cf Remark \ref{rmk:outline}). In the proof of the following lemma, we'll refer to the embedded thickenings $\mathscr{A}'$ and $\mathscr{A}''$ from the previous section.
\begin{lem}
\label{lem:Dsigmadelta}
Assume that $f$ is $r$--connected $(r \geq 1)$. Then the square $\mathscr{D}_{\sigma,\delta}$ is $0$--cartesian provided that $\dim(K,L) \leq k \leq n-3$ and $n \geq 2(k-r)+3$.
\end{lem}
\begin{proof}
Using the projections $C_K \times D^1 \rightarrow C_K$ and $X \times D^1 \rightarrow X$, write the outer squares associated with $\mathscr{A}'$ and $\mathscr{A}''$ as follows:
\begin{equation*}
\xymatrix@!0{
& &  A'_K\ar[rr]\ar[dd] & & C_K\ar[dd] & & & & A''_K\ar[rr]\ar[dd] & & \US_XC_K \ar[dd] 
\\ 
\\
& &  \US_XK \ar[rr] & & X & & & & K \ar[rr] & & X  
}
\end{equation*}
The maps constructed in Lemmas \ref{lem:sectionC} and \ref{lem:sectionK} and the fact that the squares above are cocartesian provide the following weak equivalences:
\[
\tag{$1$}
\begin{split}
\label{eqn:we1}
X \cup_{C_K}A'_K \simeq (\US_XK\cup_{A'_K}C_K)\cup_{C_K}A'_K\simeq \US_XK\cup_{A'_K}A'_K\simeq\US_XK
\end{split}
\]
\[
\tag{$2$}
\begin{split}
\label{eqn:we2}
X \cup_KA''_K \simeq (\US_XC_K\cup_{A''_K}K)\cup_KA''_K\simeq \US_XC\cup_{A''_K}A''_K\simeq\US_XC_K
\end{split}
\]
\begin{eqnarray}\label{eqn:we3}
\US_X C_K \cup_X \US_X K &=& X \cup_{C_K} X\cup_X \US_X K\\ 
                                                                                  &\simeq& X \cup_{C_K} \US_X K \nonumber\\
                                                                                    &\simeq& X \cup_{C_K}A'_K\cup_{A'_K}\US_X K \nonumber\\ 
                                                                                   &\simeq& (\US_XK\cup_{A'_K}C_K)\cup_{C_K}A'_K\cup_{A'_K}\US_X K \nonumber\\
                                                                                   &\simeq& \US_XK\cup_{A'_K}A'_K\cup_{A'_K}\US_X K \nonumber\\                                                                                   
                                                                                   &\simeq& \US_XK \cup_{A'_K} \US_XK \nonumber\\
                                                                                  &=&\US_{\US_XK}A'_K \nonumber
\end{eqnarray}
What we have so far can be combined to form the following punctured $4$--cube, in which the weak equivalences (\ref{eqn:we1}), (\ref{eqn:we2}), and (\ref{eqn:we3}) above imply that every 2--dimensional face that meets $\US_{\US_XK}A'_K$ is cocartesian.
\begin{equation*}
\xymatrix@!0{
& 
& & & C_K \ar[ddd]\ar[ld]
& & & C_K \ar[ld]\ar[rrr]\ar[ddd]
& & & A'_K \ar[ld]\ar[ddd]
\\
K \ar[rrr]\ar[ddd]
& & & X \ar[ddd]
& & & X \ar[rrr]\ar[ddd]
& & & \US_X K\ar[ddd]
\\
& & & & \ar[rr]
& & & &
\\
& K \ar[rrr]\ar[ld]
& & & X \ar[ld]
& & & X \ar[rrr]\ar[ld]
& & & \US_X K\ar[ld]
\\
A''_K \ar[rrr]
& & & \US_X C_K
& & & \US_X C_K \ar[rrr]
& & & \US_{\US_X K} A'_K
}
\end{equation*}
Let $B$ denote the homotopy limit of this punctured $4$--cube, so that we have a cartesian $4$--cube
\begin{equation*}
\xymatrix@!0{
& B \ar[rrr]\ar[ld]\ar'[d][ddd]
& & & C_K \ar[ddd]\ar[ld]
& & & C_K \ar[ld]\ar[rrr]\ar[ddd]
& & & A'_K \ar[ld]\ar[ddd]
\\
K \ar[rrr]\ar[ddd]
& & & X \ar[ddd]
& & & X \ar[rrr]\ar[ddd]
& & & \US_X K\ar[ddd]
\\
& & & & \ar[rr]
& & & &
\\
& K \ar[rrr]\ar[ld]
& & & X \ar[ld]
& & & X \ar[rrr]\ar[ld]
& & & \US_X K\ar[ld]
\\
A''_K \ar[rrr]
& & & \US_X C_K
& & & \US_X C_K \ar[rrr]
& & & \US_{\US_X K} A'_K
}
\end{equation*}
\begin{rmk}
To avoid technical difficulties, we will assume that we have mapped the original punctured cube to a new punctured cube by a pointwise weak equivalence, and that the limit of the new punctured cube is the homotopy limit of the original punctured cube. The new punctured cube, together with its limit, is a strictly commutative cube. Hence, we will assume that the $4$--cube above is strictly commutative, but will keep the notation as is.
\end{rmk}
\nd To complete the proof of the lemma, we wish to apply the 4D Face Theorem to this cube, so we check its hypotheses. As noted above, the cube is cartesian. Since every $2$--dimensional face which meets $\US_{\US_X K} A'_K$ is cocartesian, \cite[Definition 2.5]{G} implies that every $3$--dimensional face which meets $\US_{\US_X K} A'_K$ is strongly cocartesian. A quick check shows that, in the notation of the 4D Face Theorem, $k_1 = k_2 = n-k-1$ and $k_3 = k_4 = r$. In particular, since $k \leq n-3$, we have $k_1,k_2 \geq 2$. Hence, the theorem applies and as a consequence, we have that $B$ is connected and that the square
\begin{equation*}\tag{$\mathscr{B}$}
\begin{split}
\xymatrix@!0{
B \ar[rr]\ar[dd] & & K\ar[dd]
\\
\\
K \ar[rr] & & A''_K
}
\end{split}
\end{equation*}
is $(2(n-k-1) + 2r -1) = (2n-2k+2r-3)$--cocartesian.
\begin{clm}
\label{clm:cocartesianreplacement}
There exists a space $A$ and a $(2n-2k+2r-5)$--connected map $A \rightarrow B$ such that the square
\begin{equation*}
\xymatrix@!0{
A \ar[rr]\ar[dd] & & K\ar[dd]
\\
\\
K \ar[rr] & & A''_K
}
\end{equation*}
(with $B$ replaced by $A$) is cocartesian.
\end{clm}
\begin{proof}[Proof of Claim] Our hypotheses imply that $2(n-k+r) - 3 \geq 3$. Now, choose a basepoint for $B$ (which, in turn, bases $K$ and $A''_K$). Since $A''_K \rightarrow K$ is $(n-k)$--connected, we have a long exact sequence on cohomology with respect to any local coefficient system on $A''_K$ given by
\begin{equation*}
\cdots \rightarrow H^{*-1}(K \vee K) \rightarrow H^*(\overline{A''_K}, K \vee K) \rightarrow H^*(A''_K) \rightarrow \cdots
\end{equation*}
By definition, $A''_K$ is a PD space of dimension $n$. Since $k \leq n-3$, the sequence above implies that the relative cohomology of $K \vee K \rightarrow A''_K$ vanishes in degrees $\geq n+1$. That is, the relative cohomology of $K \vee K \rightarrow A''_K$ vanishes in degrees $> 2(n-k+r)-3$ provided that $n +1 \leq 2n-2k+2r-2$, which is equivalent to $n \geq 2(k-r)+3$. Hence \textit{cocartesian replacement} (\cite[Theorem 4.2]{K2}) applies and gives the desired space $A$. 
\end{proof}
\nd Now consider one of the other $2$-dimensional faces of the $4$-cube:
\begin{equation*}
\xymatrix@!0{
B\ar[rr]\ar[dd] & & C_K\ar[dd]
\\
\\
K \ar[rr] & & X
}
\end{equation*}
Replacing $B$ with the space $A$ constructed in the previous claim, form the square
\begin{equation*}
\xymatrix@!0{
A\ar[rr]\ar[dd] & & C_K\ar[dd]
\\
\\
K \ar[rr] & & X
}
\end{equation*}
Using \cite[Claims 6.5 and 6.6]{K2} we infer that this square is the outer square associated with a relative embedded thickening of $f$. This gives the desired object of $E_f(K,X\ \rel\ L)$ and, hence, the desired $0$--cell of $\E_f(K,X\ \rel\ L)$. That is, $\mathscr{D}_{\sigma,\delta}$ is $0$--cartesian.
\end{proof}

\begin{proof}[Proof of Theorem A]
By Lemma \ref{lem:Dsigmadelta} the square
\begin{equation}\tag{$\mathscr{D_{\sigma,\delta}}$}
\begin{split}
\label{diagram:Dsigmadelta}
\xymatrix@!0{
\E_f(K,X\ \rel\ L) \ar[rrrrrrr]^<<{\qquad\qquad \sigma}\ar[dd]_{\delta} & & & & & & &  \E_{\US_Xf}(\US_XK, X \times D^1\ \rel\ \US_X L)\ar[dd]^{\delta}
\\
\\
\E_{f^1}(K,X \times D^1\ \rel\ L) \ar[rrrrrrr]_{\hspace*{-0.3in}\sigma} & & & & & & & \E_{\US_Xf^1}(\US_X K, X \times D^2\ \rel\ \US_X L)
}
\end{split}
\end{equation}
is $0$--cartesian provided that $k \leq n-3$ and $n \geq 2(k-r) +3$. That is, under these assumptions we have a $0$--connected map 
\begin{equation*}
\E_f(K,X\ \rel\ L) \xrightarrow{\alpha} P
\end{equation*}
where $P$ denotes the homotopy limit of the partial diagram gotten from $\mathscr{D}_{\sigma,\delta}$ by considering only the bottom horizontal and right vertical maps. Repeated application of the decompression map on both sides of $\mathscr{D}_{\sigma,\delta}$ forms an infinite tower of embedding spaces, and it follows from \cite[Corollary 3.4]{K5} that this process produces a weak equivalence after passing to colimits. Thus, by a downward induction on codimension, we may assume that the bottom horizontal map in $\mathscr{D}_{\sigma,\delta}$ is $j$--connected for some $j \geq 0$. This implies that the map $\beta$ in the diagram
\begin{equation*}
\begin{split}
\label{diagram:Dsigmadelta}
\xymatrix@!0{
\E_f(K,X\ \rel\ L) \ar[drrr]^-{\alpha}\ar[rrrrrrr]^<<{\qquad\qquad \sigma}\ar[dd]_{\delta} & & & & & & &  \E_{\US_Xf}(\US_XK, X \times D^1\ \rel\ \US_X L)\ar[dd]^{\delta}
\\
& & & P\ar[rrru(0.7)]^-<<{\beta}\ar[dll(0.7)]
\\
\E_{f^1}(K,X \times D^1\ \rel\ L) \ar[rrrrrrr]_{\hspace*{-0.3in}\sigma} & & & & & & & \E_{\US_Xf^1}(\US_X K, X \times D^2\ \rel\ \US_X L)
}
\end{split}
\end{equation*}
is $j$--connected, from which it follows that $\sigma$ is $0$--connected. That is, the stabilization map induces a $\pi_0$ surjection provided that $k \leq n-3$ and $n \geq 2(k-r)+3$. Using Remark \ref{rmk:deltaset}, we conclude that $\sigma$ induces a surjection
\begin{equation*}
\pi_j(\E_f(K,X\ \rel\ L)) \rightarrow \pi_j(\E_{\US_X f}(\US_X K, X \times D^1\ \rel\ \US_X L))
\end{equation*}
provided that $j \leq n-2(k-r)-3$. This gives the desired surjectivity statement. For the injectivity statement, assume that we are given two PD embeddings $\mathscr{D}_0$ and $\mathscr{D}_1$ with underlying maps $f_0$, $f_1\: K \rightarrow X$ and consider the associated embedding $\mathscr{D}_{K \amalg K}$ with underlying map $f_0 \amalg f_1\: K \amalg K \rightarrow \partial(X \times D^1)$. Further, assume that $f_0$ and $f_1$ give rise to the same embedding with underlying map $\US_X K \rightarrow X \times D^1$ after applying the stabilization map. Note that this embedding is relative to the embedding $\mathscr{D}_{K\amalg K}$. Then we have an associated map of pairs
\begin{equation*}
F:(K \times D^1, K \amalg K) \rightarrow (X \times D^1, \partial(X \times D^1)).
\end{equation*}
Assume that $n \geq 2(k-r)+c$ for some constant $c$. This is equivalent to $r \geq 2k - n + (c-r)$. According to \cite[Corollary B]{K5} $\mathscr{D}_0$ is concordant to $\mathscr{D}_1$ provided that $c-r \leq 3$. But we have assumed that $r \geq 1$, so $c$ is at least $4$. Hence the induced map
\begin{equation*}
\pi_j(\E_f(K,X\ \rel\ L)) \rightarrow \pi_j(\E_{\US_X f}(\US_X K, X \times D^1\ \rel\ \US_X L))
\end{equation*}
is injective provided that $j \leq n-2(k-r)-4$. This completes the proof of Theorem \hyperref[thmA]{A}.
\end{proof}

If we now fix an object $K \xrightarrow{f} S^n$ in $\T(S^n)$ and recall that there is a map
\begin{equation*}
\E_{\US_{S^n} f}(\US_{S^n} K , S^n \times D^1\ \rel\ \US_{S^n}\emptyset) \xrightarrow{c} \E_{\US f}(\US K, S^{n+1})
\end{equation*}
given by collapsing out the copies of $S^n$ on either end of $S^n \times D^1$, then the proof of Theorem \hyperref[thmB]{B} follows immediately from Theorem \hyperref[thmA]{A} and the following lemma.  
\begin{lem}
\label{lem:cisconnected}
Assume that $\dim(K) \leq k \leq n-3$. Then the ``collapse'' map $c$ is $0$--connected.
\end{lem}
\begin{proof}
Let
\begin{equation}\tag{$\mathscr{D}_{\US f}$}
\begin{split}
\xymatrix@!0{
A\ar[rr]\ar[dd] & & C\ar[dd]
\\
\\
\US K \ar[rr]_{\US f} & & S^{n+1}
}
\end{split}
\end{equation}
denote a vertex of the space $\E_{\US f}(\US K, S^{n+1})$. To lift $\mathscr{D}_{\US f}$ back to a vertex of $\E_{\US_{S^n} f}(\US_{S^n} K , S^n \times D^1\ \rel\ \US_{S^n}\emptyset)$, it will be enough to solve the lifting problem
\begin{equation*}
\xymatrix@!0{
& & \US_{S^n} K \ar[dd]^q
\\
\\
A\ar[rr]\ar@{-->}[rruu] & & \US K. 
}
\end{equation*}
Obstruction theory tells us that the problem has a solution provided that $\dim(A) \leq \conn(q)$. Now, by definition $A$ is a PD space of dimension $n$ and, thus, is cohomologically $n$--dimensional. By \cite[Proposition 8.1]{GK1} we infer that $\dim(A) \leq n$ (this uses the assumption that $k \leq n-3$, which implies that $n \geq 2$). Thus, it will be enough to show that $\conn(q) = n$. But this follows easily from the fact that $S^n \rightarrow \ast$ is $n$--connected. 
\end{proof}

\section{Appendix: Comparison with the Goodwillie--Klein Model}
\label{sec:appendix}

In the series of papers \cite{GK1} and \cite{GK2}, Goodwillie and Klein develop machinery that provides a lift from the embedding spaces in the  Poincar\'e category to the embedding spaces in the smooth category. In this appendix, we will sketch an argument that shows that our space of embeded thickenings is related, via a homotopy fiber sequence, to the version that Goodwillie and Klein use in \cite{GK1} and \cite{GK2}.

\subsection*{A Category Whose Realization Models the Space $\E_f(K,X)$} For simplicity, we will only consider the case $L=\emptyset$, and construct a space of embedded thickenings for $f$ in the sense of Definition \ref{dfn:PDembedding}. To this end,  let $K$ and $X$ be as in Definition \ref{dfn:PDembedding}. We will permit ourselves a slight abuse of notation here by writing $(X,K)$ to denote the fact that there is a cofibration $f\: K\to X$ which is not necessarily an inclusion. Then we have a map of homotopy finite cofibration pairs
\[
\phi=(\phi_1,\phi_0):(\partial X,\emptyset) \rightarrow (X,K).
\]
Let $\T(\phi)$ be the category whose objects are homotopy finite cofibration pairs $(C,A)$ that factorize the map $\phi$. A morphism in $\T(\phi)$ is just a map of pairs $(C,A)\rightarrow (C',A')$ which is compatible with factorizations. Such a morphism is a weak equivalence in $\T(\phi)$ if each of the associated maps $C\to C'$ and $A\to A'$ is a weak homotopy equivalence. Let $w\T(\phi)$ denote the full subcategory of $\T(\phi)$ whose only morphisms are weak equivalences. Notice then that specifying an object $(C,A)$ of $w\T(\phi)$ is precisely the same thing as specifying a diagram
\begin{equation}\tag{$\mathscr{D}$}
\begin{split}
\xymatrix@!0{
A \ar[rr]\ar[dd] & & C \ar[dd] & & \partial X \ar[ll]
\\
\\
K \ar[rr]_{f} & & X
}
\end{split}
\end{equation}
 
\begin{dfn}
The \textit{category of embedded thickenings} of $f\: K \to X$, denoted $ET_f(K,X)$ , is the full subcategory of $w\T(\phi)$ in which the diagram displayed above is an embedded thickening in the sense of Definition \ref{dfn:PDembedding}.
\end{dfn}
\nd It is then evident that the realization of the nerve of $ET_f(K,X)$ is a model for the space $\E_f(K,X)$ up to homotopy. Furthermore, we have a forgetful functor
\[
F\: w\T(\phi) \rightarrow w\T(\phi_0)
\]
given by $(C,A)\mapsto A$. This forgetful functor can be thought of as a ``pullback functor" in the sense of \cite[Section 2]{GK1}. In particular, the forgetful functor sends an object $(C,A)$ of $w\T(\phi)$ to the object $(C,A)\times_{(X,K)}(\emptyset,K) = (\emptyset,A)$ of the category of factorizations of the map $(\emptyset,\emptyset)\to(\emptyset,K)$, which is the same thing as specifying the object $A$ of $w\T(\phi_0)$. According to \cite[Proposition 2.11]{GK1}, we then have a homotopy fiber sequence (after realizing nerves in the following sequence) given by
\[
w\T(\phi'_{A_0})\to w\T(\phi)\xrightarrow{F}w\T(\phi_0).
\]
where, for a given object $A_0\in w\T(\phi_0)$, the category $w\T(\phi'_{A_0})$ consists of pairs $(C_1,A_1)$ that factorize the map $\phi'_{A_0}\: (\partial X,\emptyset) \amalg (\emptyset,A_0)\to (X,K)$, and such that the map $A_0\to A_1$ is a weak homotopy equivalence. 

Assume now that $K$ comes equipped with an $(n-1)$-dimensional Poincar\'e boundary $\partial K$, ie, that $(K,\partial K)$ is an $n$-dimensional PD pair. Then $\partial K$ is an object of $w\T(\phi_0)$ and an unraveling of definitions shows that the right Quillen fiber (comma category) $\partial K\backslash F$ has objects given by diagrams of the form $(\mathscr{D})$ above, along with a choice of weak homotopy equivalence $\partial K\to A$. That is, the category $\partial K\backslash F$ is just the category $w\T(\phi'_{\partial K})$. The space $E^h(K,X)$ of Poincar\'e embeddings used in \cite{GK1} and \cite{GK2} is precisely the (realization of the nerve of the) full subcategory of $w\T(\phi'_{\partial K})$ whose objects are given by diagrams of the form $(\mathscr{D})$ that are embedded thickenings in the sense of Definition \ref{dfn:PDembedding}, along with the requirement that the weak equivalence $\partial K\to A$ is the identity map. The ``thicker" version of the category $\partial K\backslash F$, in which we do not require the weak equivalence $\partial K\to A$ to be the identity, will give a space homotopy equivalent to the Goodwillie--Klein model after realization of nerves. The discussion above establishes the following:

\begin{prop} Let $T_n(K) \subset w\T(\phi_0)$ be the full subcategory with objects $A\to K$ such that $(\overline{K},A)$ is an $n$--dimensional PD pair. Then we have a homotopy fiber sequence
\[
E^h(K,X)\to |ET_f(K,X)| \to |T_n(K)|
\]
\end{prop}
\nd where $E^h(K,X)$ is the homotopy fiber taken at $\partial K\in T_n(K)$.

\end{document}